\documentclass[12pt, oneside,]{article}  	
\usepackage[bottom = 4cm, top = 3.2cm, left = 3cm, right = 3cm]{geometry}

\usepackage{amsrefs}

\usepackage{graphicx}				
\usepackage{yfonts}
\usepackage{tikz}
\usepackage{caption}
\usepackage{amsmath,amsfonts,amssymb,amsthm,epsfig,epstopdf,titling,url,array,mathtools}
\usepackage{sidecap}
\usepackage{float}
\usepackage[utf8]{inputenc}
\usepackage[english]{babel}
\usepackage{hyperref}
\usepackage{t1enc}
\usepackage{subcaption}
\usetikzlibrary{positioning,arrows,calc}								
\usepackage{wrapfig, framed, caption}
\usepackage[titletoc]{appendix}
\usepackage{enumitem}
\usepackage{faktor}
\usetikzlibrary{decorations.pathmorphing}

\def\Uf{\ensuremath\mathrm{Uf}}

\def\uf{\ensuremath\mathfrak{F}^{\mathfrak{ue}}}
\def\ue{\ensuremath\mathfrak{ue}}

\def\dom{\ensuremath\text{dom}}

\def\>{\ensuremath\rangle}
\def\<{\ensuremath\langle}

\def\tp{\ensuremath\mathsf{tp}}

\newtheorem{theorem}{\bfseries Theorem}[section]
\newtheorem{proposition}[theorem]{\bfseries Proposition}
\newtheorem{lemma}[theorem]{\bfseries Lemma}
\newtheorem{claim}[theorem]{\bfseries Claim}
\newtheorem{fact}[theorem]{\bfseries Fact}
\newtheorem{corollary}[theorem]{\bfseries Corollary}
\theoremstyle{definition}  
\newtheorem{remark}[theorem]{\bfseries Remark}
\newtheorem{definition}[theorem]{\bfseries Definition}
\newtheorem{observation}[theorem]{\bfseries Observation}
\newtheorem{example}[theorem]{\bfseries Example}

\newtheorem{manualtheoreminner}{Theorem}
\newenvironment{manualtheorem}[1]{%
	\renewcommand\themanualtheoreminner{#1}%
	\manualtheoreminner
}{\endmanualtheoreminner}

\usepackage{amssymb}

\usepackage{authblk}
\author[]{Zalán Molnár\thanks{Eötvös Loránd University, Department of Logic, Budapest, Hungary\\ \text{  } \text{  }\text{   } \text{  }molnar.zalan.agoston@btk.elte.hu}} 

\title{Notes on ultrafilter extensions of almost bounded structures}
\date{}

\let\<=\langle
\let\>=\rangle

\let\models=\vDash

\renewcommand{\setminus}{\smallsetminus}

\def\cA{\mathcal{A}}
\def\cB{\mathcal{B}}
\def\cC{\mathcal{C}}
\def\cD{\mathcal{D}}

\def\cF{\mathcal{F}}

\def\cH{\mathcal{H}}

\def\cN{\mathcal{N}}

\def\cP{\mathcal{P}}

\def\cS{\mathcal{S}}

\def\Uf{\ensuremath\mathrm{Uf}}

\def\uf{\ensuremath\cA^{\mathfrak{ue}}}
\def\ue{\ensuremath\mathfrak{ue}}

\def\dom{\ensuremath\text{dom}}

\usepackage{tikz}
\usetikzlibrary{decorations.pathmorphing}
\usetikzlibrary{positioning,arrows,calc}
\tikzset{
	modal/.style={>=stealth’,shorten >=1pt,shorten <=1pt,auto,node distance=1.5cm,
		semithick},
	world/.style={circle, draw,minimum size=.1cm,fill=gray!15},
	point/.style={circle,draw,inner sep=0.3mm,fill=black},
	circ/.style={circle,draw,inner sep=0.4mm,fill=white},
	reflexive above/.style={->,loop,looseness=7,in=120,out=60},
	reflexive below/.style={->,loop,looseness=7,in=240,out=300},
	reflexive left/.style={->,loop,looseness=7,in=150,out=210},
	reflexive right/.style={->,loop,looseness=7,in=30,out=330}
}

\def\Uf{\ensuremath\mathrm{Uf}}

\def\uf{\ensuremath\cA^{\mathfrak{ue}}}
\def\ue{\ensuremath\mathfrak{ue}}

\def\dom{\ensuremath\text{dom}}

\begin{document}
	
	\maketitle
	
	\begin{abstract}
		\noindent  We extend the results of \cite{molnar} and contribute to the problems of \cite{Saveliev} by studying the connections  between the two constructions: the one coming from modal logic, called ultrafilter extensions, and ultrapowers. Throughout we restrict ourselves to relational structures with one binary relation. In  \cite{molnar},  a class  called bounded structures with a universal finite bound on the maximal in-and out-degree was studied. It was shown that members of this class can be elementarily embedded to their ultrafilter extensions.  Comparing  ultrafilter extensions and ultrapowers, it seems that the real challenge is  when  the degree has no  global finite bound, or there are elements with infinite degree. This is a first step towards this direction by slightly relaxing the notion of boundedness, called almost bounded structures. Among others, we show that members of this class are still elementary substructures of their extensions,  elementary embeddings can be lifted up to the extensions, moreover for  countable case the extensions are isomorphic to certain ultrapowers. We also comment on the modal logics they generate.
	\end{abstract}
	\noindent \textbf{keywords:} ultrafilter extension,  modal logic, almost bounded structures
	\section{Overview}
	\subsection{Background and notation}
	\paragraph{Background.} This paper mainly extends the results of \cite{molnar} and contributes to  the problems related to \cite{Saveliev}. In general, we investigate the  model theoretic features of  ultrafilter extensions, the construction which emerged  from  modal logic and the theory of Boolean algebras with operations (BAO).  Many of the most intriguing problems in modal logic somehow revolve around the notion of ultrafilter extension 
	or its algebraic counterpart,  nowadays called canonical extension. Just to name the two most important:  Fine's conjecture \cite{fine,goldblatt1} and canonicity \cite{fine,goldblatt2}. These extensions were  implicitely introduced  in Jónsson--Tarski \cite{jonsson}, finalized by R. Goldblatt in \cite{goldblatt} and J. van Benthem in \cite{Benthem,correspondence}. Their importance  is also clear from a model theoretic approach. The analogy according to which ''\textit{ultrafilter extensions are to modal logic, what ultrapowers are for first-order logic}'' becomes  apparent when ''modalizing'' the Keisler-Shelah isomorphism theorem: Two (pointed) Kripke models are modally equivalent if and only if they have bisimilar ultrafilter extensions (cf. \cite{black}, Theorem 2.62). Throughout we leave aside the algebraic counterpart, and concentrate  on the ultrafilter extension of structures $\cA =\<A, R\>$ equipped with a single binary relation $R$. Our   interest is to understand how far the analogy can be stretched and investigate classes, where the  constructions not only have similar applications, but  actually share common model theoretic features. 
	
	In  \cite{molnar},  the class  called bounded structures was studied, in which a universal finite bound on the maximal in-and out-degree was postulated. It was shown that  members of this class enjoy the following  properties: 
	\begin{itemize}
		\item  each one   is  elementary substructure of its ultrafilter extension,
		\item elementary equivalence  is lifted up to their extensions,
		\item their modal logic  coincides with their ultrafilter extension and ultrapowers.
	\end{itemize}
	
	\noindent  In order to have   similar properties for other classes, it seems that the real challenge  is  when  the degree has no  global finite bound, or there are elements with infinite degree. This paper is a first step towards this direction by slightly relaxing the notion of boundedness. 
	\\

	\paragraph{Notation.} We shortly introduce the notation and terminology. Everything which is not covered here including notation is standard, and can be found in \cite{chang, black}. Fix $\cA=\<A, R\>$, by  $R_\infty$ we denote the set of nodes with infinite degree and  reserve the symbols $S$ and $P$ for the relations 
	\begin{align}
		P =\{\langle s,t\rangle\in R: s\in R_\infty \text{ or } t\in R_\infty\}; \qquad S= R\setminus P
	\end{align}
	
	For $Q\subseteq A$ we define the  operations  with their obvious meanings:
	\begin{align}
		\begin{split}
			Q[X] & = \{s\in A: (\exists w\in X)\,Q(w,s) \},\\
			\<Q\>(X) & = \{w\in A: (\exists s\in X)\,Q(w,s) \}.
		\end{split}
	\end{align}
	The \textit{ultrafilter  extension} of $\cA$ is the structure $\cA^\ue=\<\Uf(A), R^\ue\>$, where $\Uf(A )$ is the set of all ultrafilters over $A$ and $R^\ue$ is defined as:
	\begin{align}
		R^\ue(u,v) \overset{\text{def}}{\Leftrightarrow} \{\<R\>(X): X\in v\}\subseteq u\Leftrightarrow  \{R[X]: X\in u\}\subseteq v.
	\end{align}
	
	By $\pi_w = \{X \subseteq W : w \in X\}$ we denote the principal ultrafilter generated by $w$, while $\Uf^*(A)$ denotes the set of non-principal ultrafilters. The map $w\mapsto \pi_w$ is an embedding of $\cA$ into $\cA^\ue$. Throughtout, we consider the basic modal language with a single unary connective $\diamondsuit$, and $\Box$ abbreviates $\neg\diamondsuit\neg$. Modal formulas are interpreted in $\cA$ in the usual way: If $V: \mathsf{Prop}\to \cP(A)$ is a valuation of the propositional variables and $w\in A$, then $\<\cA,V\>, w\Vdash \varphi$ means that the modal formula $\varphi$ is true at $w$ in the model $\<\cA,V\>$.
	The only non-obvious case is 
	\begin{align}
		\<\cA,V\>, w\Vdash\diamondsuit \varphi\Leftrightarrow (\exists w\in A)R(w,v) \text{ and }\<\cA,V\>, w\Vdash\varphi. 
	\end{align}
	Finally, $\cA\Vdash\varphi$ means validity, that is $\<\cA,V\>, w\Vdash \varphi$,  for each $w\in A$ and  valuation $V$. $\Lambda(\cA) =\{\varphi: \cA\Vdash\varphi\}$ is called the modal logic of $\cA$, which can be seen as a fragment of  monadic second-order logic via the standard translation. It is well known that $\Lambda(\cA^\ue)\subseteq \Lambda(\cA)$. For canonical formulas the reverse  containment hold, but it is undecidable in general,  which  formulas (either modal or first-oder) are preserved under ultrafilter extensions (cf. \cite{black}).

	\begin{definition}
		A structure $\cA =\<A, R\>$ is called \textit{bounded}, if the maximal in-and out-degree of $\cA$ is universally bounded by some finite number. $\cA$ is \textit{almost bounded}, if except finitely many nodes with infinite degree, $\cA$ is bounded.
	\end{definition}
	
	\begin{remark}
		
		This seemingly innocent twist already calls for  attention. Bounded structures are  always generated substructures of their extensions (cf. \cite{molnar}, Theorem 3.5) in the sense of \cite{black}, Definition 3.13. On the other hand, for the almost bounded case a node with infinite degree will be related to many non-principal ultrafilters, resulting intricate connections between the ''new'' and ''old'' elements. However, we can always  have a full control over these connections through a natural decomposition of the  relation.
	\end{remark}
	\subsection{Results of the paper}
	
	Our  results are proven in Section 2. Some of the constructions are further generalizations of the methods introduced in \cite{molnar}, while decomposing the original relation into  suitable components. 
	
	\begin{manualtheorem}{\ref{thm:almost_bounded}}
		\textit{Let  $\cA$ be an almost bounded structure with $|A|=\lambda$.  If the ultrapower $\cA^*$ is at least  $2^\lambda$-regular , then $\cA\preceq\uf\preceq\cA^*$.}
	\end{manualtheorem}

	\begin{manualtheorem}{\ref{cor:elementary}}
		\textit{If $\cA_1$ and $\cA_0$ are almost bounded and $\cA_0\preceq \cA_1$, then $\cA^\ue_0\preceq \cA^\ue_1$.}
	\end{manualtheorem}

	\noindent	Releasing the condition on boundedness  to almost bounded structures, $\cA$ and $\cA^\ue$ might generate different modal logics. However, it coincides with the one generated by the ultrapowers. 
	\begin{manualtheorem}{\ref{thm:almost_modal}}
		\textit{Let $\cA$ be an almost bounded structure,  $\cA^*$ be an ultrapower  over some $\aleph_1$-incomplete ultrafilter. Then  $\Lambda(\uf)=\Lambda((\uf)^\ue) =\Lambda(\cA^*)$.}
	\end{manualtheorem}
	\noindent	As a consequence, if $\cA,\cB$ are almost bounded, then $\cA\equiv \cB$ implies $\Lambda(\cA^\ue)=\Lambda(\cB^\ue)$. If instead, we were assuming simply boundedness, $\cA\equiv \cB$ implies $\Lambda(\cA)=\Lambda(\cB)$.
	Thus, we have examples for all the interesting cases of classes $\mathsf{K}$ such that for each member $\cA\in \mathsf{K}$ we have
	\begin{enumerate}
		\item[(a)] $\cA \equiv \cA^\ue$, and $\Lambda(\cA)=\Lambda(\cA^\ue)$ (cf. bounded structures \cite{molnar}),
		\item[(b)]  $\cA \equiv \cA^\ue$, but $\Lambda(\cA)\neq\Lambda(\cA^\ue)$ (cf. Example \ref{emp}),
		\item[(c)] $\cA \not\equiv \cA^\ue$, but $\Lambda(\cA)=\Lambda(\cA^\ue)$ (limit ordinals $\alpha\geq \omega^2$ with $\<\alpha, \leq \>$).
	\end{enumerate}
	To see (c) consider the following example. For an ordinal $\alpha$ we let $\cA_{\alpha} = \<\alpha, \leq\>$. It is folklore (e.g.  \cite{benthem}) that $\Lambda(\cA_{\omega^2})=\mathbf{S4.3}$ and for every limit ordinal $\alpha \geq \omega^2$ we have $\Lambda(\cA_{\alpha})=\Lambda(\cA_{\omega^2})$. Indeed, by Bull's Theorem $\mathbf{S4.3}$ is finitely approximable, thus for each $\varphi\not\in \mathbf{S4.3}$ there is a finite $\mathbf{S4.3}$ frame $\cF_\varphi\not\Vdash \varphi$. Obviously $\mathbf{S4.3}\subseteq \Lambda(\cA_\alpha^\ue)$.  It is routine to check that  each such $\cF_\varphi$ is a $p$-morphic image of $\cA_{\alpha}^\ue$,  consequently $\Lambda(\cA_{\alpha}^\ue)= \mathbf{S4.3}$. On the other hand, $\cA_{\alpha}\not\equiv \cA_{\alpha}^\ue$. This is because  $\cA_{\alpha}\not\models \exists x\forall y (Rxy\to Rxy)$, whereas $\cA_{\alpha}^\ue\models \exists x\forall y (Rxy\to Rxy)$, since $\cA_{\alpha}^\ue$ has a ``final'' cluster of non-principal ultrafilters.

	\begin{manualtheorem}{\ref{sat}}
		\textit{Let $\cA$ be an almost bounded structure,  $\cA^*$ be an ultrapower over some $\kappa$-regular ultrafilter.  Then  $\cA^\ue$ is at least  $2^{2^{\aleph_0}}$, while $\cA^*$ is at least  $2^\kappa$-saturated.}
	\end{manualtheorem}	
	\noindent In particular, whenever $\cA$ is countable,  the index set is $|I| = 2^{\aleph_0}$ and $\cA^*$ is an ultrapower modulo some regular ultrafilter over $I$, then $\cA^\ue \cong \cA^*$. However, isomorphism cannot be achieved for larger cardinals in general, as Example \ref{ex:2} shows.

	\section{Almost bounded structures}
	
	\subsection{Coherent decomposition}

	We start with some basic observations regarding  ultrafilter extensions of almost bounded structures. From \cite{molnar}, Proposition 4.2 and 4.3 imply the following observation:
	\begin{observation}\label{obs_degree}
		Let $\cA$ be an almost bounded structure,  then	
		\begin{enumerate}\itemsep-1pt
			\item[(i)] Both $\<A,S\>$ and $\<\Uf(A), S^\ue\>$ are bounded structures,
			\item[(ii)] $R^\ue(u,u)$ \  iff\  $\{w\in A: R(w,w)\}\in u$.
		\end{enumerate}
	\end{observation}

	Our goal is to show that every almost bounded structure $\cA$ is an elementary substructure of $\cA^\ue$. Instead of working directly with $R$, we decompose it into $S$ and $P$, and extend the language by constant symbols $d_w$, for each $w\in R_\infty$. The interpretations of the constants in $\cA$, $\uf$ or in an ultrapower are given in the obvious way.   Throughout  $\cA^* =\< A^*, R^*\>$ abbreviates any ultrapower of $\cA$ modulo some ultrafilter $\cD$ over an infinite index set $I$. Its elements are denoted by $[a]_\cD\in A^*$, for $a\in\, ^IA$. We let
	
	\begin{enumerate}\itemsep-2pt
		\item[•] $\cA^\circ =\< A, S, P, d_w \>_{w\in R_\infty} $,
		\item[•]  $\cA^\flat= \<\Uf(A), S^\ue, P^\ue, d_w\>_{w\in R_\infty}  =(\cA^\circ )^\ue$, 
		\item[•] $\cA^\sharp= \< A^*, S^*, P^*, d_w \>_{w\in R_\infty} = (\cA^\circ )^*$. 
	\end{enumerate}\index{$\cA^\bullet, \cA^\flat, \cA^\sharp$}
	The next proposition tells us that the decomposition is coherent in the sense that it commutes with their extensions.  
	
	\begin{proposition}\label{prop:iso}
		Let $\cA=\<A, R\>$ be an almost bounded structure with $s,t\in A$. If $u,v\in \Uf^*(A)$, then
		\begin{enumerate}\itemsep-1pt
			\item[(i)] 	$R^\ue(\pi_s,\pi_t) \text{ iff }  S^\ue(\pi_s,\pi_t) \text{ or }  P^\ue(\pi_s,\pi_t)$,
			\item[(ii)]  	$R^\ue(\pi_s,u) \text{ iff } P^\ue(\pi_s,u)$, similarly $ R^\ue(u,\pi_s)  \text{ iff }  P^\ue(u,\pi_s)$,
			\item[(iii)]  	$R^\ue(u,v) \text{ iff } S^\ue(u,v)$.
		\end{enumerate} 
		Moreover, exactly the same properties hold for  ultrapowers.
	\end{proposition} 
	\begin{proof}
		(i) is immediate, we only argue for (ii)-(iii).
		Since both  the reducts $\<A, S\>$ and $\<\Uf(A), S^\ue\>$ are bounded, we must have $\<\pi_s,u\>\not \in  S^\ue$ and $\<u, \pi_s\>\not \in  S^\ue$,  for all $u\in \Uf^*(A)$. Thus, whenever $R^\ue(\pi_s,u)$ (or $R^\ue(u,\pi_s)$), we must have $s\in R_\infty$. On the other hand $P^\ue \subseteq R^\ue$ by the monotonicity of $(\cdot)^\ue$. This proves (ii).\\
		
		\noindent	For (iii) fix $u,v\in \Uf^*(A)$. As $R_\infty$ is finite we have $\<u,v\>\not\in P^\ue$, and $S^\ue \subseteq R^\ue$ by monotonicity.  Assume that $R^\ue(u,v)$.  By Observation \ref{obs_degree} (i), $u$ must have a finite degree, hence $ S^\ue(u,v)$. In more details: suppose $X\in v$, then $X' =\{w\in X: w\not\in R_\infty\}\in v$. Since $\<S\>(X')=\<R\>(X')\in u$ by assumption, and  $\<R\>(X')= \<S\>(X)$. Hence we obtain $S^\ue(u,v)$, since $X$ was  arbitrary. This proves (iii).
	\end{proof}
	
	\noindent	This  implies that $S^\ue \cup P^\ue = (S\cup P)^\ue = R^\ue$ and $S^* \cup P^* = (S\cup P)^* = R^*$. We introduce the translation $(\cdot)^\sharp : L_{\{R, d_w\}_{w\in R_\infty}}\to L_{\{S,P,d_w\}_{w\in R_\infty}}$ between the  languages:
	\begin{align}
		\begin{split}
			(\tau_1=\tau_2)^\sharp := \tau_1 = \tau_2; \quad
			R(\tau_1,\tau_2)^\sharp& := S(\tau_1,\tau_2) \vee P(\tau_1,\tau_2); \\
			(\varphi \vee \psi)^\sharp := (\varphi)^\sharp\vee (\psi)^\sharp;\quad
			(\neg \varphi)^\sharp& :=\neg (\varphi)^\sharp;\quad
			(\exists x \varphi)^\sharp :=\exists x (\varphi)^\sharp
		\end{split}
	\end{align}
	\begin{proposition}\label{lem_transfer}
		Let $\cA$ be an almost bounded structure, $\cA^* $ be any of its ultrapower.  For all formulas $\varphi(\overline{x})$ of $L_{\{R,d_w\}_{w\in R_\infty}}$ and $[\overline{a}]_\cD\in A^*$, $\overline{u}\in \Uf(A)$:
		
		\begin{enumerate}
			\item[(i)] $\cA^*\models \varphi([\overline{a}]_\cD)$ iff $\cA^\sharp\models \varphi([\overline{a}]_\cD)^\sharp,$
			\item[(ii)]  $\cA^\ue\models \varphi(\overline{u})$ iff $\cA^\flat\models \varphi(\overline{u})^\sharp.$
		\end{enumerate}

	\end{proposition}\label{lem:ultra}
	\begin{proof} Follows from Proposition \ref{lem_transfer} by a straightforward induction.
	\end{proof}

	\begin{proposition}\label{prop_transfer}
		Let $\cA$ be an almost bounded structure, $\cA^* $ be any of its ultrapower. If $\cA^\flat\preceq \cA^\sharp$, then $\uf\preceq \cA^*$.
	\end{proposition}
	
	\begin{proof}
		We use  Tarski--Vaught criterion. Let $u_0,\dots, u_n\in \Uf(A)$ and suppose that for an $L_R$-formula $\varphi(x, x_0, \dots, x_n)$ we have  $\cA^*\models\exists x\varphi (u_0,\dots, u_n)$. Then 
		\begin{align*}
			\cA^*\models\exists x\varphi (u_0,\dots, u_n) &\Leftrightarrow \cA^\sharp \models(\exists x\varphi (u_0,\dots, u_n))^\sharp \text{ by Proposition \ref{lem_transfer}}\\
			&\Leftrightarrow \cA^\sharp\models\exists x(\varphi (u_0,\dots, u_n))^\sharp\\
			&\Rightarrow \cA^\sharp\models \varphi (u, u_0,\dots, u_n)^\sharp \text{ for some $u\in \Uf(A)$}, \text{ as } \cA^\flat\preceq \cA^\sharp\\
			&\Leftrightarrow \cA^*\models \varphi (u, u_0,\dots, u_n).
		\end{align*}
	\end{proof}

	\subsection{A \L o\'s's Lemma type property}
	We start with a couple of definitions. 
	Fix a structure $\cA=\langle A,R\rangle$, $s,t\in A$, and a sequence of elements $\overline{w}=\<w_0,\dots, w_n\>$ from $A$.  The set $R[\overline{w}]$ is called an $R$-\textit{road} from $s$ to $t$ if it is the set containing exactly one of $R(w_i,w_{i+1})$ or $R(w_{i+1}, w_i)$ for $i<n$, moreover  $w_0= s$ and $w_n=t$.  Let $u_0,\dots,u_n$ be different ultrafilters over $A$. We say that the pairwise disjoint subsets $D_{u_0} , \dots , D_{u_n} \subseteq A$ are \textit{distinguishing sets} for the ultrafilters $u_0,\dots,u_n$, if $D_{u_j}\in u_i$ exactly when $j=i$.
	
	\begin{definition}\label{def:uf_road}
		Fix $\cA=\<A,R\>$, $Q\subseteq R$ and let $\overline{w}$ be a $Q^\ue$-road of length $n$ from $u$ to $v$.  The \textit{ultrafilter road of $Q^\ue[\overline{w}]$ based on $X\in u$} is the set $\Delta(X, Q^\ue [\overline{w}])$ defined inductively by:
		\begin{align}
			\begin{split}
				\Delta_0\big(X, Q^\ue [\overline{w}]\big) &= X\\
				\Delta_{i+1}\big(X,Q^\ue [\overline{w}] \big) &= \begin{cases}
					Q\big[\Delta_i\big(X, Q^\ue [\overline{w}] \big) \big]\cap D_{w_{i+1}} & \text{if } Q^\ue(w_i,w_{i+1})\in Q^\ue[\overline{w}],\\ 	
					\\
					\big\<Q\big\>\big(\Delta_i\big(X, Q^\ue [\overline{w}] \big) \big)\cap D_{w_{i+1}}   & \text{if } Q^\ue(w_{i+1},w_{i})\in Q^\ue[\overline{w}].
				\end{cases}
			\end{split}
		\end{align}
		Finally, we let $\Delta\big(X,Q^\ue [\overline{w}]\big) := \Delta_{n}\big(X,Q^\ue [\overline{w}]\big)$. 
	\end{definition}
	
	\noindent	To illustrate this definition, consider the case when $R^\ue[w_0,w_1,w_2]$ is the road $w_0R^\ue w_1(R^\ue)^{-1}w_2$ and $X\in u$. Then $\Delta(X, R^\ue[\overline{w}])=\<R\>\big(R[X]\cap D_{w_1}\big)\cap D_{w_2}$

	\begin{proposition}\label{uf_road}
		Let $\cA=\< A,R\>$ be a structure, $Q\subseteq R$, $u,v\in \Uf(A)$ and $\overline{w}$ be a $Q^\ue$-road from $u$ to $v$. Then,  $\Delta\big(X, Q^\ue[\overline{w}]\big)\in v$.
	\end{proposition}
	
	\begin{proof}  Readily follows by  induction on the length of the road.  
	\end{proof}
	
	Let $w\in A$, $Q\subseteq A\times A$ and $n\in\omega$. By $\<w\>_Q^n$ we denote the substructure given by the elements  reachable at most $n$ steps from $w$ along $Q$ or  $Q^{-1}$, and call it the $(Q,n)$-neighborhood of $w$. For $\cA_i = \<A_i,R_i\>$ with $w_i\in A_i$ and $Q_i\subseteq A_i\times A_i$, where $i\in \{0,1\}$ we use the notation $\<w_0\>^n_{Q_0}\cong \<w_1\>^n_{Q_1}$ to indicate the existence of an isomorphism $f: \<w_0\>^n_{Q_0} \to \<w_1\>^n_{Q_1}$ such that $f(w_0) = w_1$.

	\begin{definition}
		Let $\cA =\<A,R\>$ be an almost bounded structure and $\cA^\times, \cA^\diamond\in \{\cA, \cA^\ue, \cA^*\}$. For $w\in A^\times$, $u\in A^\diamond$ and $n\in \omega$ we say that $\<w\>_{S^\times}^n$ and $\<u\>_{S^\diamond}^n$ are \textit{$P$-isomorphic} (notation: $\<w\>_{S^\times}^n \cong_{P} \<u\>_{S^\diamond}^n$) if the following conditions hold:
		\begin{enumerate}\itemsep-2pt
			\item[•] there is an  $f:  \<w\>_{S^\times}^n \cong \<u\>_{S^\diamond}^n$,
		\end{enumerate}
		and  for all $s\in \<w\>_S^n$ and  $t\in R_\infty$ we have 
		\begin{enumerate}\itemsep-2pt
			\item[•]$P^\times(s,t)$ if and only if $ P^\diamond(f(s),t)$,
			\item[•]  $P^\times(t,s)$ if and only if $ P^\diamond(t,f(s))$.
		\end{enumerate}
	\end{definition}
	

	In \cite{molnar}, Theorem 4.9   proves a local variant of \L o\'s's Lemma  for bounded structures  regarding the neighborhoods of each element:
	\begin{fact}\label{thm2}
		Let $\cA=\<A, R\>$ be a bounded structure. Then for all $n\in \omega$ and $u\in \Uf(A)$ 
		\[\{w\in A: \exists f_{w,n}: \langle u\rangle^n_{R^\ue} \cong \langle w\rangle^n_{R} \}\in u, \]
		moreover, $f_{w,n}(v_i)\in D_{v_i}$, for all  $v_i\in \<u\>^n_{R^\ue}$. 
		
	\end{fact}
	
	\noindent We generalize this property to almost bounded structures.

	\begin{lemma}\label{thm}
		Let $\cA=\<A,R\>$ be an almost bounded structure. Then for all $n\in \omega$ and $u\in \Uf(A)$ we have
		\[V_{n} = \{w \in A: \exists f_{w,n}: \<u\>^n_{S^\ue}\cong_{P}\<w\>^n_{S}\}\in u,\]
		moreover, $f_{w,n}(v_i)\in D_{v_i}$, for all  $v_i\in \<u\>^n_{S^\ue}$. 
	\end{lemma}
	\begin{proof} 
		Instead of $\cA$, we consider   $\cA^\circ =\< A,S,P, d_w\>_{w\in R_\infty}$. Since the reduct $\<A, S\>$  is   bounded,  by Fact \ref{thm2} we obtain:
		\begin{align}
			U_n = \{w\in A: \exists f_{w,n}:\<u\>^n_{S^\ue}\cong\<w\>^n_{S}\}\in u,\label{U_n}
		\end{align}	
		\noindent with $f_{w,n}(v_i)\in D_{v_i}$, where the $D_{v_i}$'s are the distinguishing sets for the members from $v_i\in \<u\>^n_{S^\ue}$. Throughout we work with a fixed function $f_{w,n}$, for each $w\in U_n$.  Suppose,  for some $v_i\in \<u\>^n_{S^\ue}$ we have $ R^\ue( \pi_s,
		v_i)$. By Proposition \ref{prop:iso}  $s\in R_\infty$. We claim that
		\begin{align}
			U_{\pi_s}^{v_i}  =\{w\in U_n : P(s,f_{w,n}(v_i)) \}\in u. 
		\end{align}
		Otherwise $A\setminus 	U_{\pi_s}^{v_i}\in u$.
		Fix some road $S^\ue[\overline{w}]$ from $u$ to $v_i$. Then 
		\begin{align}
			\Delta(A\setminus 		U_{\pi_s}^{v_i}  , S^\ue[\overline{w}] )\in v_i,
		\end{align}
		by Proposition \ref{uf_road}. By (\ref{U_n}) each $t\in \Delta(A\setminus 		U_{\pi_s}^{v_i}  , S^\ue[\overline{w}] )$ is of the form $f_{w,n}(v_i)$, for some $w\in A\setminus 	U_{\pi_s}^{v_i}$, whence
		\begin{align}
			s\not\in \<P\>\big(\Delta(A\setminus 	U_{\pi_s}^{v_i} , S^\ue[\overline{w}])\big).
		\end{align}
		This contradicts to the fact that $ P^\ue(\pi_s, v_i)$, hence contradicting to the assumption that $ R^\ue(\pi_s, v_i)$. Similar argument can be conducted if  $\<\pi_s,v_i\>\not \in R^\ue$ or $\<v_i, \pi_s\>\in R^\ue$ or $\<v_i,\pi_s\>\not\in R^\ue$.  For each $s\in R_\infty$ and $v_i\in \<u\>^n_{S^\ue}$ we may repeat the whole process depending on how the possible pairs are related to each other  in order to obtain: 
		\begin{align}
			W_n := \bigcap_{v_i\in \<u\>^n_{S^\ue}}\bigcap_{s\in R_\infty}U_{\pi_s}^{v_i} \in u.
		\end{align}
		Finally,  $W_{n} \subseteq V_n\in u$, as desired.
	\end{proof}

	\subsection{Comparing the constructions}
	Fix an almost bounded structure $\cA$ and $\<u\>^n_{S^\ue}$, for some $u\in \Uf(A)$ and $n\in \omega$. For each $v_i,v_j\in \< u\>_{S^\ue}^n$  consider distinct variables $x_i,x_j$. We reserve  $x_0$ for $u$ and assume  $|\< u\>_{S^\ue}^n| = k$.  Define   $\Sigma$ to be the set of formulas:
	\begin{align}
		\begin{split}
			\{S (x_i,x_j)  :\<v_i, v_j \>\in  S^\ue \}&\cup  \{\neg S (x_i,x_j): \<v_i, v_j \>\not \in  S^\ue\} \\
			\left\{ P (d_w, x_i) : \begin{array}{l}
				\<\pi_w, a_i \> \in  P^\ue, \text{ and } \\ w\in R_\infty\end{array} \right\} &\cup \left\{\neg P (d_w, x_i) : \begin{array}{l} \<\pi_w, a_i \> \not\in  P^\ue,  \text{ and } \\ w\in R_\infty\end{array}\right\}\\
			\left\{ P (x_i, d_w):\begin{array}{l} \<v_i,\pi_w \> \in  P^\ue, \text{ and }\\ w\in R_\infty\end{array} \right\} &\cup \left\{\neg P(x_i,  d_w) : \begin{array}{l} \<v_i,\pi_w \> \not\in  P^\ue,  \text{ and } \\ w\in R_\infty\end{array}\right\}\\
			\{\bigwedge_{i<k} d_w\neq x_i : w\in R_\infty \}&\cup 	
			\{\forall y\big( \bigvee_{m\leq n} \psi_m (x_0,y)\to \bigvee_{i<k} y= x_i \big)\} \\ &\cup
			\{\bigwedge_{i\neq j<k} x_i\neq x_j \},	
		\end{split}
	\end{align}
	where $\psi_n(x,y)$ is the formula stating the existence of an $n$-road along  $S$,  defined  as
	\begin{align}
		\begin{split}
			\psi_0(y_0,y_n) &:= y_0 = y_n,\\
			\psi_{n}(y_0,y_n)&:= \exists y_1\dots \exists y_{n-1}\big(\bigwedge_{0\leq  i\neq j\leq n}y_i\neq y_j  \wedge \bigwedge_{i<n} \big(S (y_i, y_{i+1}) \vee S(y_{i+1},y_i)\big)\big).
		\end{split}
	\end{align}
	To use less indices,  let $x:= x_0$ and   $\chi_n(x) := \exists x_1\dots \exists x_k\bigwedge\Sigma$. 
	Define the type
	\begin{align}
		\tp(u) =\{\chi_n(x): n\in \omega \}.
	\end{align}
	\noindent Similarly, we will consider the type $\tp([a]_\cD)$  in any ultrapower. The next lemma contains some basic properties regarding regular ultrapowers, however we give the  sketch of the arguments.

	\begin{lemma}\label{lemma:26}

		Let $\cA$ be an almost bounded structure, fix an ultrapower $\cA^*$ over an ultrafilter, which  is at least $\kappa$-regular for some infinite $\kappa$.  Then: 
		\begin{enumerate}\itemsep-1pt
			
			\item[(i)] If some $[a]_\cD\in A^*$ realizes $\tp(u)$  for  $u\in \Uf(A)$, then $\<[a]_\cD\>_{S^*}^\omega\cong_{P} \< u\>_{S^\ue}^\omega $.
			\item[(ii)]\textit{For each non-diagonal element $[a]_\cD\in A^*$ there are at least $2^{2^{\aleph_0}}$-many elements $u_\alpha\in \Uf(A)$ for which
				$  \< u_\alpha\>^\omega_{S^\ue}  \cong_{P} \langle [a]_D\rangle^\omega_{S^*}$, and $\<[u_\alpha]_\cD\>^\omega_{S^*}\neq \<[u_\beta]_\cD\>^\omega_{S^*}$, for $\alpha\neq \beta<2^{2^{\aleph_0}}$.}

			\item[(iii)] For each non-diagonal element $[a]_\cD\in A^*$, there are at least $2^\kappa$-many elements $[b_\alpha]_\cD\in  A^*$, for which $\<[b_\alpha]\>^\omega_{S^*}\cong_P \<[a]\>^\omega_{S^*}$, and $\<[b_\alpha]_\cD\>^\omega_{S^*}\neq \<[b_\beta]_\cD\>^\omega_{S^*}$, for $\alpha\neq \beta<2^\kappa$.
		\end{enumerate}
	\end{lemma}
	\begin{proof} (i) Fix some non-diagonal $[a]_\cD$ realizing $\tp(u)$. In this case $u$ realizes $\tp([a]_\cD)$. For each $n\in \omega$ and $\chi_n\in \tp([a]_\cD)$ by \L o\'s's Lemma  we get 
		\begin{align}\label{gamma}
			\gamma(n) =\{i\in I: \cA^\circ\models\chi_n(a(i))\}\in \cD.
		\end{align}
		Since $\cD$ is $\aleph_1$-incomplete there is a sequence $\< I'_n : n\in \omega\>$ such that $I'_n\supseteq I'_{n+1}$ with $I'_n\in \cD$ and $\bigcap_{n<\omega}I'_n =\emptyset$.  Consider the decreasing sequence
		\begin{align}\label{I_n}
			I_0 = I'_0; \qquad I_{n+1}= I'_n\cap \gamma(n).
		\end{align}
		
		Then,  for all $i\in I$, there is a $\max(i)$ such that $i\in I_{\max(i)}\setminus I_{\max(i)+1}$. 	For each $n\in \omega$ we let $J_n = \{i\in I: \max(i)= n\}$.  Since for every $j\in J_n$ the set $\<a(j)\>^{n}_{S}$ is finite, say has $n_j\in \omega$ elements, we can pick choice functions $a_{n_1}, \dots, a_{n_{j-1}}: J_{n} \to A$ such that for each $j\in J_n$ we have $\<a(j)\>^n_S=\{a(j), a_{n_1}(j),\dots, a_{n_{j-1}}(j)\}$ and there is a map 
		\begin{align}
			f_{n,j} : \<u\>^{n}_{S^\ue}\cong_P \<a(j)\>_S^n. 
		\end{align}
		Pick $z\in \< u\>^\omega_{S^\ue}$, we  define  the corresponding choice function $a_z: I\to A$ as follows:
		
		\begin{align}
			a_z(i)=
			\begin{cases}
				f_{\max(i),i}(z) &\text{ if }   z\in \<v\>^{\max(i)}_
				{S^\ue} \\
				b & \text{ otherwise}, 
			\end{cases}
		\end{align}
		where $b$ is arbitrary, but fixed. Then the map $\eta:\<[a]_\cD\>^\omega_{S^*} \to \<u\>^\omega_{S^\ue}$ defined by $\eta([a_z]_\cD)= z$ satisfies $\<[a]_\cD\>^\omega_{S^*} \cong_P \<u\>^\omega_{S^\ue}$.	\\
		
		\noindent (ii) Let $[a]_\cD$ be non-diagonal.
		For each $n\in \omega$ let $\widehat{\gamma(n)} = \{a(i): i\in \gamma(n)\}$, where $\gamma(n)$ is from (\ref{gamma}).  Since $[a]_\cD$ is non-diagonal, each $\widehat{\gamma(n)}$ is infinite. Therefore, we can
		consider an injective  choice function $f:\omega\to \bigcup_{n\in\omega }\widehat{\gamma(n)}$, that is $f(n)\in \widehat{\gamma(n)}$, and for all $n\neq m$ we have $f(n)\neq f(m)$. Hence $Y=\{f(n): n\in \omega\}$ is infinite, thus there are  $2^{2^{\aleph_0}}$ ultrafilters over $Y$. For each $v^*\in  \Uf(Y)$,  let $v\in \Uf(A) $  be the extension of $v^*$, i.e. $v^*\subseteq v$ and set $X=\{v\in \Uf(A): v^*\in \Uf(Y)\}$. By Lemma \ref{thm},  each $v\in X$ realizes $\tp([a]_\cD)$, and thus $[a]_\cD$ must realize $\tp(v)$. Therefore, by (i) we have $\<[a]_\cD\>^\omega_{S^*} \cong_P \<v\>^\omega_{S^\ue}$. Finally, by construction $|X| = |\Uf(Y)|$.\\

		\noindent (iii) Fix $[a]_\cD\in A^*$.  Consider the decreasing sequence from (\ref{I_n}). For all $i\in I$ let  $i\in I_{\max(i)}\setminus I_{\max(i)+1}$ and define the set
		\begin{align}
			A_{\max(i)} = \{w\in A: \<w\>^{\max(i)}_S\cong_P \<a(i)\>^{\max(i)}_{S} \}.
		\end{align}
		
		Since $[a]_\cD$ is non-diagonal, for each $i\in I$ we can fix $B_i\subseteq A_{\max(i)}$ with $|B_i|=\aleph_0$. Then $|\prod_{i\in I}B_i/\cD| \geq 2^\kappa$, as $\cD$ is at least $\kappa$-regular. Also,  each $[b]_\cD \in \prod_{i\in I}B_i/\cD$ realizes  $\tp([a]_\cD)$: Let $\chi_n\in \tp([a]_\cD)$, we claim that
		\begin{align}
			\gamma(n)\subseteq \{i\in I: \cA^\circ\models\chi_n(b(i)) \}\in \cD,
		\end{align}
		where $\gamma(n)$ is from (\ref{gamma}). Indeed, since $j\in I_{\max(j)}$, thus $n\leq \max(j)$ by construction. This implies  $\<b(j)\>^n_S \cong_P \<a(j)\>^n_S$, consequently $j\in  \{i\in I: \cA^\circ\models\chi_n(b(i)) \}$. Finally, similarly to (i), one can show that whenever some $[b]_\cD$ realizes $\tp([a]_\cD)$, then $\<[b_\cD]\>^\omega_{S^*}\cong_P \<[a]_\cD\>^\omega_{S^*}$.
	\end{proof}
	
	Recall the elementary facts that whenever $\cA\leq \cB$ and for each finite sequence $\overline{a}\in A$ and $b\in B$ if there is an automorphism fixing $\overline{a}$, while moving $b$ into $A$,  then $\cA\preceq\cB$. Also, whenever $\cA\preceq \cC$ and $\cA\leq \cB$ with $\cB\preceq \cC$, then $\cA\preceq\cB$. We make use of these without any further reference.
	
	\begin{theorem}\label{thm:almost_bounded}
		Let  $\cA$ be an almost bounded structure with $|A|=\lambda$.  If the ultrapower $\cA^*$ is  $2^\lambda$-regular , then $\cA\preceq\uf\preceq\cA^*$.
	\end{theorem}
	\begin{proof}
		Let $\cA=\langle A,R\rangle$ be an almost bounded structure with $\lambda = |A|$, and set $\kappa =  2^{2^\lambda} = |\Uf(A)|$. Consider its decomposition $\cA^\circ =\< A, P, S, d_w\>_{w\in R_\infty}$ and let $\cA^\sharp= \langle A^*, P^*, S^*, d_w \rangle_{w\in R _\infty}$  be an (at least) $2^\lambda$-regular ultrapower of $\cA^\circ$. It is enough to show that $\cA^\circ\preceq \cA^\flat\preceq\cA^\sharp$, since by Proposition \ref{prop_transfer} it  follows that $\cA\preceq\cA^\ue \preceq \cA^*$. First, we argue for the embedding $\cA^\flat \leq\cA^\sharp$. 
		By Lemma \ref{lemma:26} we get
		\begin{align}
			\langle w\rangle^\omega_S\cong_P\langle \pi_w\rangle^\omega_{S^\ue}\cong_P \langle [c_w]_\cD\rangle^\omega_{S^*},
		\end{align}
		where $[c_w]_\cD$ is the image of $w$ under the diagonal embedding.
		Hence, we let  $f_0:\cA^\flat\to\cA^\sharp$ to be the partial $P$-isomorphism,  where $f_0(\pi_w) = [c_w]_\cD$. Consider an enumeration $\{u_\gamma\in \Uf^*(A): \gamma <\kappa\}$.
		For $\gamma <\kappa$ we construct  a sequence of  partial $P$-isomorphisms $f_\gamma : \cA^\flat\to\cA^\sharp$, such that:
		\begin{enumerate}\itemsep-2pt
			\item[($a$)] $f_\alpha \subseteq f_{\gamma}$, for $\alpha < \gamma$,
			\item[($b$)] $\langle u_\gamma\rangle^\omega_{S^\ue}\subseteq \mathrm{dom}(f_{\gamma+1})$, 
			\item[($c$)] $|\mathrm{dom}(f_\gamma)|  \leq |\gamma| +\aleph_0$,
			\item[($d$)] $\langle u_\gamma\rangle^\omega_{S^\ue} \cong_P \langle f_{\gamma+1}(u_\gamma)\rangle^\omega_{S^*}$,
			\item[($e$)] $f_\gamma = \bigcup_{\alpha<\gamma}f_\alpha$ if $\gamma$  is limit.
		\end{enumerate}
		Assume for some $\beta<\kappa$ we have already defined the  map $f_\beta$, so let us define $f_{\beta+1}$.  There are two cases: if $u_\beta \in \mathrm{dom}(f_\alpha)$ for some $\alpha < \beta$, then let $f_{\beta +1} = f_\beta$. Otherwise $u_\beta \not\in \bigcup_{\alpha<\beta}\dom(f_\alpha)$. In this scenario  no $u\in\langle u_\beta\rangle^\omega_{S^\ue}$ is mapped by any previously defined $f_\alpha$ to $\cA^\sharp$.     By Theorem \ref{thm} $\tp(u_\beta)$ is finitely satisfiable in $\cA^\circ$, hence in $\cA^\sharp$. Since $\cA^\sharp$ is $\lambda$-regular, it is $\lambda^+$-saturated  over types having parameters from $[A]^{<\lambda}$. As $\tp(u_\beta)$ is a type in $ \cS^{\cA^\sharp}(R_\infty)$,  there is some $[b]_\cD\in A^*$ realizing it. Let $X =\bigcup_{\alpha<\beta} \<f_{\alpha+1}(u_\alpha)\>^\omega_{S^*}$, then $|X| \leq |\beta|\cdot\aleph_0 < \kappa$.
		Hence, by Lemma \ref{lemma:26} we can find an element $[a]_\cD\in A^*$ such that $\<[u_\beta]_\cD\>^\omega_{S^\ue}\cong_P\<[a]_\cD\>^\omega_{S^*}\cong_P\<[b]_\cD\>^\omega_{S^*}$ with $\<[a]_\cD\>^\omega_{S^*}\cap X=\emptyset$. Denote the $P$-isomorphism from $\<[u_\beta]_\cD\>^\omega_{S^\ue}$ to $\<[a]_\cD\>^\omega_{S^*}$ by $\eta$, and define $f_{\beta+1}$ as
		\begin{align}
			f_{\beta+1}= \bigcup_{\alpha\leq \beta} f_\beta \cup \{\<u,\eta(u)\>: u\in \<u_\beta\>^\omega_{S^\ue}  \}.
		\end{align}
		
		It is easy to see that $f_{\beta+1}$ indeed satisfies the conditions ($a$)-($e$). Finally, we set $f = \bigcup_{\gamma <\kappa}f_\gamma$. Then $f$ is an embedding with domain $\Uf(A)$. It is left to show that $\cA^\flat\preceq\cA^\sharp$. Pick $f(u_0),\dots, f(u_n)\in A^*$ and $[a]_\cD\in A^*$.  By Lemma \ref{lemma:26}  there will be $v_0,\dots, v_{n+1}\in \Uf(A)$ such that for $\ell,i\leq n+1$ we have $\langle f(v_i)\rangle^\omega_{S^*}\cong_P \langle [a]_\cD\rangle^\omega_{S^*}$, and  $\langle f(v_i)\rangle^\omega_{S^*}\neq \langle f(v_\ell)\rangle^\omega_{S^*}$. Hence, for  at least one such $v_i$ we must have  $\langle f(v_i)\rangle^\omega_{S^*}\neq \langle f(u_j)\rangle^\omega_{S^*}$, for  $ i\neq j\leq n$. 
		Finally, let $h$  be the automorphism that $P$-isomorphically permutes each element of $\langle [a]_\cD\rangle^\omega_{S^*}$ into $\langle f(v_i)\rangle^\omega_{S^*}$, and fixes every other point-wise. Since $h$ moves $[a]_\cD$ into $f[\cA^\flat]$,  we obtain that  $\cA^\circ\preceq\cA^\flat\preceq\cA^\sharp$. Finally, by Proposition \ref{prop_transfer} we get $\cA\preceq\cA^\ue\preceq\cA^*$.
	\end{proof}

	\subsection{On elementary embeddings}
	Under a different notion of ultrafilter extension, \cite{Saveliev} showed that  elementary embeddings cannot be lifted up to ultrafilter extensions in general, and asked for what classes   this property is transferable. In our case we have a similar result.

	\begin{proposition}\label{elementary_ext}
		Elementary substructures do not transfer to their extensions in general.
	\end{proposition}
	\begin{proof}
		Consider $\cN =\< \mathbb{N}, <^\cN\>$ with the standard ordering and fix any non-isomorphic elementary extension $\cA=\<A, <^\cA\>$. Thus, there are $w\in A$ such that  $\mathbb{N}\subseteq \<<^\cA\>(w)$. Fix $u\in \Uf^*(A)$ with $\mathbb{N}\in u$.
		Then $u<^{\cA^\ue} u$ and $u<^{\cA^\ue} \pi_w$, but $\pi_w\not <^{\cA^\ue} \pi_w$.
		Thus, for
		$\varphi^* :=\exists x \exists y (x < x \wedge x < y \wedge  y \not< y) $
		we obtain $\cN^\ue\not\models\varphi^*$, while $\cA^\ue\models \varphi^*$. Hence $\cN^\ue\not\preceq\cA^\ue$.
	\end{proof}

	\begin{remark}
		We note that the construction can be adopted to the ultrafilter extensions defined in the sense of \cite{Saveliev}: For $u,v\in \Uf(A)$, define the relation
		\begin{align}
			\widetilde{R}(u,v) \Leftrightarrow \{x\in A: \{y\in A: R(x,y)\}\in v\}\in u.
		\end{align}
		We can  significantly  simplify the main result (Theorem 1.5) of \cite{Saveliev}, stating that elementary substructures cannot be lifted up to their extensions, in general. We use the  language  of a single binary relation instead of the more complicated one from \cite{Saveliev}.
	\end{remark}
	\begin{proposition}\label{prop:transfer}
		In the language containing  a single binary relation symbol,  $\cA\preceq\cB$ does not necessarily imply $\widetilde{\cA}\preceq\widetilde{\cB}$.
	\end{proposition}
	\begin{proof} As before, fix a non-isomorphic elementary extension $\cA$ of $\cN$.  Observe that $\cN^\ue \cong \widetilde{\cN}$, thus $\widetilde{\cN}\not\models\varphi^*$. To show  $\widetilde{\cA}\models\varphi^*$,  let  $w\in A$ be an element with infinitely many $<$-predecessors. Fix $u\in \Uf^*(A)$ with $\mathbb{N}\in u$. Then  $u \widetilde{<}u$, since for any $n\in \mathbb{N}$, the set $<\![n]= \{v: n<v\}\in u$, as $<\![n]\cap \mathbb{N}$ is cofinite in $\mathbb{N}$. Hence
		$\{x\in A: \{y\in A: x<y\}\in u\}\in u$.   Since $n<w$ for each  $n\in \mathbb{N}$, it follows that $u \widetilde{<}\pi_w$. Also $\pi_w \widetilde{\not<}\pi_w$,  thus $\widetilde{\cA}\models \varphi^*$.
	\end{proof}

	\begin{theorem}\label{cor:elementary}
		If $\cA_1$ and $\cA_0$ are almost bounded and $\cA_0\preceq \cA_1$, then $\cA^\ue_0\preceq \cA^\ue_1$.
	\end{theorem}
	\begin{proof}
		Fix $\cA_i=\<A_i, R_i\>$ for $i\in \{0,1\}$ satisfying the given condition. Consider their decompositions $\cA_i^\circ =\<A_i, S_i, P_i, d_w\>_{w\in R_{{i}_\infty}}$. Then $\cA_0^\circ \preceq \cA_1^\circ$. From the duality between BAO's and relational structures (e.g. Theorem 5.47 of \cite{black}) it follows  that $\cA_0^\flat\leq \cA_1^\flat
		$. We show that the map is elementary, which implies $\cA_0^\ue\preceq \cA_1^\ue$. Pick $v\in \Uf(A_1)$ and $u_0,\dots, u_n\in \Uf(A_0)$. Let $\kappa = \max\{2^{2^{|A_0|}}, 2^{2^{|A_1|}}\}$.  By the Keisler--Shelah Isomorphism Theorem we  find isomorphic ultrapowers $\cA^\sharp_0=\<A^*_0, S^*_0,P^*_0, d_w\>_{w\in R_{0_\infty}}$ and $\cA_1^\sharp=\<A^*_1, S^*_1,P^*_1, d_w\>_{w\in R_{1_\infty}}$, both are $\kappa$-regular.  Using the same techniques from Theorem \ref{thm} and $\cA_0^\sharp\cong \cA_1^\sharp$, we obtain $\cA_i^\flat \preceq \cA_i^\sharp$ and some $[a]_\cD\in A_0^*$ such that $\<[a]_\cD\>_{S^*_0}^\omega \cong_P \<v\>^\omega_{S_1^\ue}$. Finally,  by Lemma \ref{lemma:26} we find $u\in \Uf(A_0)$ such that
		\begin{enumerate}\itemsep-2pt
			\item[($a$)] $\<u\>^\omega_{S^\ue_0}\cong_P \<[a]_\cD\>^\omega_{S^*_0}$,
			\item[($b$)] $\<u_i\>^\omega_{S^\ue_0}\neq \<u\>^\omega_{S^\ue_0}$ for $i\leq n$.
		\end{enumerate}
		Consequently, we can easily define  an automorphism moving $P$-isomorphically $\<v\>^\omega_{S_1^\ue}$ into $\<u\>^\omega_{S_0^\ue}$, while fixing everything else point-wise.
	\end{proof}

	\subsection{Comparing modal logics}
	As opposed to  bounded structures, we do not necessarily have modal equivalence between $\cA$ and $\cA^\ue$. The following construction provides such  classes, but can be instructive in itself.
	
	\begin{example}\label{emp}
		Let $(\ast)$ be the  (second-order) frame condition for some fixed $w\in A$:  each $u,v\in R[w]$ is connected by some finite path $uRw_1\dots w_nRv$, where  $w_i\in R[w] $ for $1\leq i\leq n$.  Define the condition $(\ast\ast)$ for some $w\in A$ as: ($\ast$) holds for $w$ \textit{and}
		\begin{align*}
			\text{ \big($w$ is irreflexive or $|R[w]|\leq 1$ or $(\exists v\in R[w]\setminus\{w\})(R(w,v)$ and  $R(v,w))$\big).}
		\end{align*} 
		Consider the  modal formulas:
		\begin{align}
			\begin{split}
				\mathsf{Alt}_n &:=\Box p_0\vee \dots \vee \Box( \bigwedge_{i=0}^{n-1} p_i \to p_{n})\\
				\varphi &:=\Big(p\wedge \neg q \wedge \square \big(p\wedge  q\to \square (p\wedge  q)\big)\wedge  \diamondsuit (p\wedge q)\Big)\to \square (p\wedge q)
			\end{split}
		\end{align}
		such that $\mathsf{Alt}_n$ and $\varphi$ share no common variables. 

		\begin{claim}
			The formula $\varphi$ (locally) corresponds to the  condition $(\ast\ast)$.
		\end{claim}	
		\begin{proof}
			$(\Leftarrow)$ Suppose $(\ast\ast)$ holds for some $\cA$ and  $w\in A$.  Fix a valuation $V$, we show $\<\cA, V\>,w\Vdash \varphi$.\\
			\textsc{Case (A):} $R[w]=\emptyset$. Then $\<\cA, V\>,w\Vdash \varphi$ holds, as $\<\cA, V\>,w\Vdash\Box(p\wedge q)$ holds.\\
			\textsc{Case (B):} $R[w] =\{v\}$, with $v\neq w$. Since $v$ is the unique successor of $w$, if $V$ is such that $\<\cA, V\>, v\Vdash p\wedge q$, then $\<\cA, V\>, w\Vdash \Box( p\wedge q)$. Thus $\<\cA, V\>,w\Vdash \varphi$.\\
			\textsc{Case (C):} $R[w] = \{w\}$. Then $\<\cA, V\>,w\Vdash \varphi$, holds as $\<\cA, V\>,w\not\Vdash (p\wedge\neg q)\wedge \diamondsuit(p\wedge q)$.\\
			So far the cases cover the possibilities for $|R[w]| \leq 1$.\\
			\textsc{Case (D):} $(\exists v\in R[w]\setminus\{w\}) R(w,v)\wedge R(v,w)$. We claim that the antecedent of $\varphi$ cannot be datisfied at $w$, thus  $\<\cA, V\>,w\Vdash \varphi$. Otherwise, suppose that indeed $V$ satisfies it. Then for some $v_0\in R[w]\setminus\{w\}$, we have $\<\cA, V\>,v_0\Vdash p\wedge q$.
			Then by $(\ast)$  there is some finite path $v_0R\dots Rv_n$ with $v_n=v$, and $w_i\in R[w]$. From $\<\cA,V\>,w\Vdash\Box(p\wedge q\to \Box (p\wedge q))$, and  $\<\cA, V\>,v_0\Vdash p\wedge q$,   it follows  by induction that  $\<\cA,V\>,v\Vdash p\wedge q$. But then $\<\cA,V\>,v\Vdash\Box( p\wedge q)$,  hence  $\<\cA,V\>,w\Vdash p\wedge q$, which is a contradiction.\\
			\textsc{Case (E):} $w$ is irreflexive, $|R[w]| > 1$ and $(\forall u\in R[w]\setminus\{w\})(R(w,u)\to \neg R(u,w))$. This case is almost the same as  above.
			Without loss of generality, assume that $V$ is such that the antecedent of $\varphi$ is satisfied at $w$. Let $v\in R[w]$, we show that $\<\cA, V\>\Vdash p\wedge q$. By assumption,  there exists a $v_0\in R[w]\setminus\{w\}$ with $\<\cA, V\>,v_0\Vdash p\wedge q$.
			Then by $(\ast)$  there is some finite path $v_0R\dots Rv_n$ with $v_n=v$, and $w_i\in R[w]$. From $\<\cA,V\>,w\Vdash\Box(p\wedge q\to \Box (p\wedge q))$, and  $\<\cA, V\>,v_0\Vdash p\wedge q$,   it follows  by induction that  $\<\cA,V\>,v\Vdash p\wedge q$.\\
			
			\noindent	($\Rightarrow$) Suppose $(\ast\ast)$ does not hold for some $\cA$ and $w\in A$.\\
			Case (A):  $w$ is reflexive moreover, there is some $v\in R[w]\setminus\{w\}$ with $w\not\in R[v]$, and  $(\forall u\in R[w]\setminus\{w\})(R(w,u)\to \neg R(u,w))$.  Consider the valuation 
			\begin{align}
				\begin{split}
					V(p)& = A,\\
					V(q) &= A\setminus \{w\}.
				\end{split}
			\end{align}
			It is easy to see that indeed $\<\cA, V\>, w\not\Vdash\varphi$. \\
			\text{Case (B):} $(\ast)$ does not hold for $w$, i.e. there are $v,u\in R[w]$ having no finite path from $v$ to $u$ consisting of elements from $R[w]$. This implies that $v\neq w$.   Consider the sets
			\begin{align}
				\begin{split}
					X_0 &= R[v] \\
					X_{n+1} &= R\big[ X_n\cap R[w]\big] .
				\end{split}
			\end{align}
			Then $w\not\in X_n$ and $w\not\in R[X_n]$, for every $n\in \omega$. By setting $V(p)= \bigcup_{n\in\omega}X_n \cup \{w,v\}$ and $V(q)=A\setminus\{w\}$, it is easy to see that $\<\cA, V\>, w\not\Vdash\varphi$, hence $\cA\not\Vdash \varphi$. 
		\end{proof}
		Let $\mathsf{K}$ be any  subclass of the following class:  $\cA$ is almost bounded, each node with infinite out-degree has $(\ast\ast)$, moreover  there is some $w\in A$ with infinite out-degree such that for each $v\in R[w]$ the set $R[w]\cap R[v]$ is finite. 	Let $n\in\omega$ be the bound on the degree   after forgetting the elements with infinite degree from $\cA$.

		\begin{claim}
			$\cA\Vdash \mathsf{Alt}_n\vee \varphi$, while $\cA^\ue\not\Vdash \mathsf{Alt}_n \vee \varphi$.
		\end{claim}
		\begin{proof}
			It is easy to see that $\cA\Vdash \mathsf{Alt}_n\vee \varphi$. Since  $\mathsf{Alt}_n$ locally corresponds to the property having out-degree at most $n$, the formula is satisfied under each valuation, under each node with finite degree. Similarly, $\varphi$ is locally valid on each node  having infinite out-degree. Nevertheless, $\cA^\ue\not\Vdash \mathsf{Alt}_n\vee \varphi$. Consider a node $w\in A$ with infinite out-degree,  such that for each $v\in R[w]$ the set $R[w]\cap R[v]$ is finite. Fix some $u\in \Uf^*(A)$ with $R^\ue(\pi_w,u)$.  Proposition \ref{prop:iso} implies that for no $v\in R[w]$ we can have $R^\ue(\pi_v,u)$, otherwise it would mean that $R[w]\cap R[v]$ is infinite. Thus there cannot be any finite path from any such $\pi_v$ to $u$. Consequently $(\ast)$ is not satisfied at $\pi_w$, hence $\varphi$ is falsified at $\pi_w$. Since $\pi_w$ has infinite out-degree, $\mathsf{Alt}_n$ is also falsified at $\pi_w$. Thus $\cA^\ue\not\Vdash \mathsf{Alt}_n \vee \varphi$, because $\varphi$ shares no common variables with $\mathsf{Alt}_n$.
		\end{proof}

	\end{example}
	
	Consequently, for each $\cA\in\mathsf{K}$ we have  $\cA\equiv\cA^\ue$, but $\Lambda(\cA)\neq \Lambda(\cA^\ue)$. It follows that $\varphi$ has no first-order correspondent, otherwise would be preserved under ultrafilter extensions (cf. \cite{Benthem}). It is natural to ask how   the modal theories are related to  ultrapowers, or the iterated application of the  extension.
	
	\begin{theorem}\label{thm:almost_modal}
		Let $\cA$ be an almost bounded structure,  $\cA^*$ be an ultrapower  over some $\aleph_1$-incomplete ultrafilter. Then  $\Lambda(\uf)=\Lambda((\uf)^\ue) =\Lambda(\cA^*)$.
	\end{theorem}
	\begin{proof} First,  note that $\aleph_1$-incomplete ultrafilters are $\aleph_0$-regular.
		The core idea is to count the possible valuations on the possibly different neighborhoods of the elements, in order to establish bisimulation between the structures.  We only  argue for $\Lambda(\uf)\subseteq\Lambda(\cA^*)$, the rest are similar with the obvious modifications.
		Suppose $\cA^*\not\Vdash \varphi$, for some modal formula $\varphi$. Thus there is a valuation $V$ and some element in $A^*$ under which $\varphi$ is falsified. In order to show $\cA^\ue\not\Vdash\varphi$, we construct a valuation $W$ on $\cA^\ue$ and a surjective bisimulation $Z$ in the sense of \cite[Definition 2.16]{black} between the models $\<\cA^\ue, W\>$ and  $\<\cA^*, V\>$. This will ensure that $\cA^\ue\not\Vdash\varphi$.  Consider an enumeration of the non-principal ultrafilters  $E = \{u_\beta: \beta< \lambda\}$ of $\cA^\ue$ and an enumeration of the non-diagonal elements  $F = \{[b_\beta]_\cD: \beta< \kappa\}$ of $\cA^*$. For $\delta < 2^{\aleph_0}$,  by a  back-and-forth argument we define  recursively $Z_\delta,Y_\delta\subseteq E$ and $T_\delta, X_\delta\subseteq F$, such that the following conditions are satisfied: 
		\begin{enumerate}\itemsep-2pt
			\item[($a$)] \textit{separation}: For each $u,v \in Y_\delta$ we have $\<u\>_{S^\ue}^\omega\neq\<v\>_{S^\ue}^\omega$, while $\<u\>_{S^\ue}^\omega\cong_{P}\<v\>_{S^\ue}^\omega$, and similarly for $X_\delta$,
			\item[($b$)] \textit{difference}: For all $\gamma < \delta$ and $u\in Y_\gamma$, $v\in Y_\delta$ we have $\<u\>_{S^\ue}^\omega\not\cong_{P}\<v\>_{S^\ue}^\omega$, and similarly for $X_\delta$,
			\item[($c$)] \textit{iso}: For each $u\in Y_\delta$ and $[b]_\cD\in X_\delta$ we have $\<u\>_{S^\ue}^\omega\cong_{P}\<[b]_\cD\>_{S^*}^\omega$,
			\item[($d$)] \textit{coherence}: There is a valuation  $V_{\delta} : \mathsf{Prop} \to \mathcal{P}\big(\bigcup_{v\in Y_{\delta}}\<v\>^\omega_{S^\ue}\big)$ such that for every $[d]_\cD\in X_{\delta}$ there is some $v\in Y_\delta$ with  submodels $\big\langle \<[d]_\cD\>^\omega_{S^*}, V\big\rangle \cong \big\langle \<v\>^\omega_{S^\ue}, V_\delta\big\rangle$,
			\item[($e$)] \textit{decreasing}: For $\gamma < \delta$ we have $Z_\delta \subseteq Z_\gamma$, and similarly $T_\delta \subseteq T_\gamma$,
			\item[($f$)] \textit{covering}: $\bigcup_{\delta<2^{\aleph_0}}\bigcup_{u\in X_\delta}\< u\>^\omega_{S^\ue} = \Uf^*(A)$.
		\end{enumerate}
		
		\noindent For $\delta=0$, we let $ X_0 = \emptyset  =Y_0$ and $Z_0= E$, $T_0 = F$.  For  $\delta$  limit,  we define $Z_\delta$ and $T_\delta$ to be the intersection of the previous stages. For technical reasons only, we set $Y_\delta = X_\delta = \emptyset$ and $V_\delta$ valuates to emptyset, at the limit case.  Suppose that for some $\delta < 2^{\aleph_0}$, where $\delta = \gamma + 2n$, we have already constructed the sets.\\
		
		\noindent \textsc{Forth (constructing $X_{\delta+1}$):} Let   $[d]_\cD\in T_\delta$ be the least  non-diagonal element from the ordering  satisfying the condition: for all $[a]_\cD\in \bigcup_{\gamma \leq\delta}X_\gamma$ we have $[d]_\cD\not\in \<[a]_\cD\>^\omega_{S^*}$.  Then $[d]_\cD$ generates an $(S^*,\omega)$-neighbourhood that was not encountered before. Let $X_{\delta+1}$ be the maximal set with the property: For all $[a]_\cD \in X_{\delta+1}$ we have $\<[a]_\cD\>_{S^*}^\omega\cong_P\<[d]_\cD\>_{S^*}^\omega$, while for all  $[a]_\cD,[b]_\cD \in X_{\delta+1}$, $\<[a]_\cD\>_{S^*}^\omega\neq\<[b]_\cD\>_{S^*}^\omega$, that is its  elements generate pairwise disjoint $(S^*,\omega)$-neighborhoods that are $P$-isomorphic to $\<[d]_\cD\>_{S^*}^\omega$. This can be obtained e.g. by transfinite recursion over $\kappa$. Thus $(a)-(b)$ holds.   Define the  equivalence relation on $\bigcup_{[a]_\cD\in X_{\delta+1}}\<[a]_\cD\>^\omega_{S^*}$ as
		\begin{align}
			[e]_\cD \sim [f]_\cD  &\Leftrightarrow
			\begin{cases}
				\Big(\exists[b]_\cD, [a]_\cD \in X_{\delta+1} \Big)\Big( 	
				[e]_\cD\in \<[b]_\cD\>^\omega_{S^*}  \text{ and } [f]_\cD\in\<[a]_\cD\>^\omega_{S^*}\Big) \\ 
				\text{} [e]_\cD\mapsto [f]_\cD \text{ and }
				\big\langle \<[b]_\cD\>^\omega_{S^*}, V\big\rangle \cong \big\langle \<[a]_\cD\>^\omega_{S^*}, V\big\rangle, 
			\end{cases}
		\end{align}
		
		\noindent where $[e]_\cD\mapsto [f]_\cD$ is the map given by the $P$-isomorphism between $\<[b]_\cD\>^\omega_{S^*}$  and $\<[a]_\cD\>^\omega_{S^*}$. Thus $[e]_\cD$ and $[f]_\cD$ belong to those $(S^*, \omega)$-neighbourhoods that are $P$-isomorphic to each other, moreover  the components are isomorphic as submodels given by $V$.  Observe that the number of different equivalence classes is at most $2^{\aleph_0}$. \\

		\noindent	\textsc{(Constructing $Y_{\delta+1}$):} This case is symmetric to the previous one. By Lemma \ref{lemma:26}  there is some non-principal $v\in \Uf^*(A)$ such that $\langle [d]_\cD\rangle^\omega_{S^*} \cong_{P} \langle v\rangle^\omega_{S^\ue}$. By assumption, for each $\gamma\leq \delta$,   conditions  ($a$) -- ($c$) are satisfied, hence  $v\not\in \bigcup_{\gamma \leq\delta}Y_\gamma$.  Thus, similarly to $X_{\delta+1}$ we can define  $Y_{\delta+1}$ to be the maximal set with the property:  For all $u \in Y_{\delta+1}$ we have $\<u\>_{S^*}^\omega\cong_P\<v\>_{S^*}^\omega$, while for all  $u,w \in Y_{\delta+1}$, $\<u\>_{S^*}^\omega\neq\<w\>_{S^*}^\omega$. The construction so far satisfies ($a$) -- ($c$).\\

		\noindent \textsc{(Constructing $V_{\delta + 1}$):} Since $|Y_{\delta+1}| \geq 2^{2^{\aleph_0}}$ and $|X_{\delta+1}/\!\sim| \leq 2^{\aleph_0}$, fix is a surjection $g: Y_{\delta+1}\to {X_{\delta+1}/\sim}$. For each $z\in Y_{\delta+1}$  fix some representative $[f_z]_\cD\in g(z)$ and a $P$-isomorphism $\eta_z: \<z\>^\omega_{S^\ue}\to \<[f_z]_\cD\>^\omega_{S^*}$ with $\eta_z(z) = [f_z]_\cD$. Define the valuation  $V_{\delta+1}$ to be:
		\begin{align}
			v\in V_{\delta+1}(p) \Leftrightarrow \eta_z(v)\in V(p),\label{valuation}
		\end{align}
		where $z\in Y_{\delta+1}$ and $v\in \<z\>^\omega_{S^\ue}$. It is easy to see that $V_{\delta+1}$ satisfies ($d$). Finally, we let 
		\begin{align}\label{Z,T}
			\begin{split}
				Z_{\delta+1} &= Z_\delta\setminus \bigcup_{u\in X_{\delta+1}}\< u\>^\omega_{S^\ue},\\
				T_{\delta+1}&=	T_\delta \setminus\bigcup_{[a]_\cD\in Y_{\delta+1}}\< [a]_\cD\>^\omega_{S^*},
			\end{split}
		\end{align}
		thus ($e$) holds.\\
		
		\noindent \textsc{Back (constructing $\delta+2$):} We can simply mirror the argument above. Consider the first $v$ from the enumeration of $Z_{\delta+1}$ such that  for every  $z \in \bigcup_{\gamma \leq\delta+1}X_\gamma$ we have $v\not\in \<z\>^\omega_{S^\ue}$.   Thus $v$ has an $(S^\ue,\omega)$-neighbourhood that was not encountered before. We may define $Y_{\beta+2}$ just how we defined $X_{\beta+1}$. Then by Lemma \ref{lemma:26}  we can find some element $[b]_\cD$ such that $\<[b]_\cD\>^\omega_{S^*}\cong_{P}\<v\>^\omega_{S^\ue}$ with $[b]_\cD\not\in \bigcup_{\gamma \leq\delta+1}\in Y_\gamma$. Similarly to $Y_{\delta+1}$ we can construct $X_{\delta+2}$, also  $V_{\delta+2}$ can be done by ``pulling back'' the valuation from $V$, just as in (\ref{valuation}). Finally, $Z_{\delta+2}$ and $T_{\delta+2}$ are given in the same way as in (\ref{Z,T}). It is clear that condition ($f$) will hold.

		Next, let us  define the valuation $W$ on $\cA^\ue$ as:
		\begin{align}
			W(p)  = \bigcup_{\delta <2^{\aleph_0}}V_\delta(p) \cup (A\cap V(p)),
		\end{align}
		where $A\cap V(p)$ is  added  to copy the valuation of the diagonal elements. Finally, define the surjective bisimulation $Z$ between $\<\cA^\ue, W\>$ and $\<\cA^*, V\>$  as:
		\begin{align}
			u Z [a]_\cD \Leftrightarrow 
			\begin{cases}(\exists v\in Y_\delta)(\exists [b]_\cD\in X_\delta) (u\in \<v\>^\omega_{S^*} \text{ and } [a]_\cD\in \<[b]_\cD\>_{S^*}^\omega)\\
				\text{such that }  \eta_v(u)\sim [a]_\cD,
			\end{cases}
		\end{align}
		for some $\delta<2^{\aleph_0}$. Surjectivity of  $Z$ implies  $\cA^\ue\not\Vdash \varphi$.  The inclusion $\Lambda(\cA^*)\subseteq \Lambda(\uf)$ can be shown by a completely analogous fashion.
	\end{proof}
	
	\begin{corollary}\label{cor}
		Suppose $\cA$ and $\cB$ are almost bounded structures. If $\cA\equiv \cB$, then $\Lambda(\cA^\ue)=\Lambda(\cB^\ue)$. 
	\end{corollary}
	\begin{proof}
		If $\cA\equiv \cB$,  by the Keisler--Shelah Isomorphism Theorem we  find isomorphic ultrapowers $\cA^*$ and $\cB^*$. Hence by Theorem \ref{thm:almost_modal}  we get $\Lambda(\cA^\ue)=\Lambda(\cA^*)=\Lambda(\cB^*) =\Lambda(\cB^\ue)$.
	\end{proof}

	\subsection{Further structural and combinatorial properties}
	
	After all, it is natural to ask how far the similarity can be pushed  between the ultrafilter extension and an ultrapower. Here we have an immediate result.

	\begin{theorem}\label{sat}
		Let $\cA$ be an almost bounded structure,  $\cA^*$ be an ultrapower over some $\kappa$-regular ultrafilter.  Then  $\cA^\ue$ is at least  $2^{2^{\aleph_0}}$, while $\cA^*$ is at least  $2^\kappa$-saturated.
	\end{theorem}		
	\begin{proof}
		Let $X\in [\Uf(A)]^{<2^{2^{\aleph_0}}}$ and fix a type $p\in \mathcal{S}^{\cA^\ue}(X)$. Take a suitably saturated ultrapower $\cA^*$ realizing $p$, such that $\cA^\ue\preceq \cA^*$. Let $[a]_\cD\in \cA^*$ realizing $p$.   By Lemma \ref{lemma:26}  we can find an element $u\in \Uf(A)$ with $\<[a]_\cD\>^\omega_{S^*}\cong_P \<u\>^\omega_{S^\ue}$ and by Proposition \ref{lem_transfer} we have that $\cA^\sharp\models p^\sharp([a]_\cD)$.  Let $\varphi^\sharp\in p^\sharp$  with quantifier depth $n$ and consider $c_1,\dots, c_k\in X$ occuring in $\varphi^\sharp$.  Using Lemma \ref{lemma:26}, it is easy to show that player II can always win the $n$-round Ehrenfeucht–Fraïssé game between $\<\cA^\ue, u, c_0,\dots, c_k\>$ and $\<\cA^\sharp, [a]_\cD, c_0,\dots, c_k\>$. As a consquence $\cA^\flat\models p^\sharp(u)$, hence by Lemma \ref{lem_transfer} $\cA^\ue\models p(u)$, as desired. Similar argument applies for the ultrapower.
	\end{proof}
	\noindent From the uniqueness of saturated structures (cf. \cite{chang}, Theorem 5.1.13) it follows that:
	\begin{corollary}
		If $\cA$ is  countable, almost bounded, then $\cA^*\cong \cA^\ue$, for any regular ultrapower with cardinality $|A^*|= |\Uf(A)|$. 
	\end{corollary}
	\noindent Nevertheless, isomorphism cannot be achieved for larger cardinalities, in general. 	
	\begin{example}\label{ex:2} Consider the structure
		$\cA=\< A \uplus B, R\>$, where $|A| = \aleph_0 < |B|$, and $R = id_{A}$. Consider any regular ultrapower $\cA^*$ such that $|\Uf(A\uplus B)| = |(A\uplus B)^*|$.  Lemma \ref{lemma:26} (iii) implies that $\cA^*$ has  $2^{2^{|B|}}$ reflexive points, whereas $\cA^\ue$ has ''only'' $2^{2^{\aleph_0}}$ reflexive points, by Observation \ref{obs_degree} (ii).
	\end{example}
	
	\noindent From this, it  follows that whenever $|A|>\aleph_0$,  $\cA^\ue$  is not necessarily $(2^{2^{\aleph_0}})^+$-saturated.  This  leads us to the following analysis. By $\#\< u\>^\omega_{S^\ue}$ let us abbreviate the cardinality of the set $$ \{v\in \Uf(A): \<u\>^\omega_{S^\ue}\cong_P \<v\>^\omega_{S^\ue}\}.$$ If $|V_n| = \lambda_n$, where $V_n = \{w\in A: \<u\>^n_{S^\ue}\cong_P \<w\>^n_{S}\}$, then $\#\<u\>^n_{S^\ue} = 2^{2^{\lambda_n}}$ by Theorem \ref{lemma:26}, thus $\#\<u\>^\omega_{S^\ue} \leq 2^{2^{\lambda_n}}$. Since $n<m$ implies $\lambda_m \leq \lambda_n$, there is an $m\in \omega$ such that $ \lambda_m = \min \{\lambda_n: n\in \omega\}$. Hence, there is an infinite descending chain $V_{k+1} \subseteq V_{k}$, $k\geq m$ such that $|V_{k}| =  \lambda_m$. We end this section by a combinatorial observation on the possible isomorphic neighbourhoods.

	\begin{proposition}\label{prop:3.4.6}
		For $u\in \Uf(A)$, if $\lambda_n=|V_n|\geq \aleph_0$, for every $n\in \omega$, then $\#\<u\>^\omega_{S^\ue}= 2^{2^{\lambda_m}}$, where $\lambda_m = \min\{\lambda_n: n\in \omega\}$.
	\end{proposition}
	
	\begin{proof}
		Fix $u\in \Uf(A)$. By the remarks above $\#\<u\>^\omega_{S^\ue} \leq 2^{2^{\lambda_m}}$,  we show $2^{2^{\lambda_m}}\leq \#\<u\>^\omega_{S^\ue}$.     For each $m\leq k\neq \ell\in \omega$ we define the  sets $S_k, Z_k \subseteq V_k$ with the following properties:
		\begin{enumerate}\itemsep-2pt
			\item[(i)]  $|Z_k|=|S_k|= \lambda_m$,
			\item[(ii)] $Z_k\cap Z_\ell=\emptyset$. 
		\end{enumerate}
		Consider a two element partition $V_m = V_{m,0} \uplus V_{m,1}$, where $|V_{m,i}| = \lambda_m$.  At least one block, say $V_{m,0}$, satisfies the property: For all $\ell> m$ we have that $|V_\ell\cap V_{m,0}| = \lambda_m$.   Let $S_{m} =  V_{m,0}$ and $Z_m = V_{m,1}$. The recursive steps are similar: If $S_{k}$ has been defined,  let $S_{k} = V_{k,0}\uplus V_{k,1}$ be a partition with $|V_{k, i}| = \lambda_m$ and  say $V_{k,1}$ be the block with $|V_\ell\cap V_{k, 1}|=\lambda_m$, for  all $\ell\geq k$. Fix $S_{k+1}= V_{k,1}$ and $Z_{k+1}= V_{k,0}$.
		
		By Hausdorff's theorem, there is an independent family of sets $\mathcal{H}_k$ over each $Z_k$ with cardinality $2^{\lambda_m}$. For each $m\leq i,j\in \omega$ we fix a bijection  $f_i^j:\mathcal{H}_i\to \mathcal{H}_j$.  For a fixed $i\geq m$ the set $\mathcal{H} =\{\bigcup_{j\in\omega\setminus m} {f_i^j}[X]: X\in \mathcal{H}_i\}$ is an independent family of sets of cardinality $2^{\lambda_m}$ over $\biguplus_{m\leq k\in\omega}Z_k$. For  $f\in 2^{\mathcal{H}}$ we let
		\begin{align}
			\hat f(X)=\begin{cases}X,&f(X)=1,\\
				\biguplus_{m\leq k\in \omega}Z_k\setminus X,&f(X)=0.\end{cases}
		\end{align}
		
		\noindent By standard arguments,  for all $f\in 2^\cH$ the set
		\begin{align}
			v^*_f =\bigl\{\bigcap_{X\in \mathcal{G}} \hat f(X): \mathcal{G}\subseteq \mathcal{H} \text{ is finite}\bigr\} \cup \bigl\{V_k: k\geq m\bigr\}   
		\end{align}
		has the F.I.P. with pairwise different bases, hence generating $2^{2^{\lambda_m}}$-many ultrafilters $v^*$ over $\biguplus_{m\leq k\in\omega}Z_k$, thus over $A$.  Any such extension $v^*\subseteq v\in \Uf(A)$ realizes $\tp(u)$ resulting $\<u\>_{S^\ue}^\omega\cong_P \<v\>^\omega_{S^\ue}$. 
	\end{proof}
	
	Intuitively speaking, the difference between the  constructions is that for ultrapowers  $\#\<[a]_\cD\>^\omega_{S^*}$ is \textit{as large as possible}, while for ultrafilter extension $\#\<u\>^\omega_{S^\ue}$ is \textit{as small as possible}.

	\section{Discussion}
	We are interested in  classes  having no  finite bound on the degree, while still preserving some of the desirable properties presented in this paper. Though locally finite structures are not elementary equivalent to their ultrafilter extensions in general (cf. \cite{molnar}), we wonder if elementary equivalence can be lifted up  to their  extensions. That is:
	\begin{enumerate}
		\item[(1)]  Does $\cA\equiv \cB$ imply $\cA^\ue\equiv \cB^\ue$, for locally finite structures? 
	\end{enumerate}

	Another promising class would be to bound the length of each path instead of bounding the degree. Fix $n\in \omega$, and let $\widetilde{R}$ be the symmetric closure of $R$.  The condition which we postulate for $\cA$ is: $$\forall x_0\dots\forall x_n(\bigwedge_{i=0}^{n-1} \widetilde{R}(x_i,x_{i+1}) \to \bigvee_{i\neq j} x_i= x_j).$$
	
	\begin{enumerate}
		\item[(2)] What properties are transferable to ultrafilter extensions for this class? 
	\end{enumerate}
	Note that,  using $R$ instead of $\widetilde{R}$, we have example for $\cA\not\equiv \cA^\ue$.

	\begin{example}
		Consider  three disjoint  sets $\mathbb{N}$,  $S= \{s_n : n\in \omega\}$, and $T =\{t_n: n\in \omega\}$. We define a structure $\cA=\<A, R\>$  as follows. For each $n\in\omega$ we let $S_n =\{n\}\times S$ and $T_n = S_n\times T $. Then let  $A= \mathbb{N}\cup \bigcup_{n\in \omega}(S_n\cup T_n)$, moreover $R$ is given by:
		\begin{align*}
			R = \{\<n, \<n, s_m\> \>: n,m\in \omega\}&\cup \{\<\<n, s_m\>, \<n, s_m, t_i\>\>: n,m,i\in\omega \} \\
			&\cup \{\<n, \<{j}, s_{m}, t_i\>\>:n,j,m,i\in \omega, n < j\}.
		\end{align*}
		In picture, $\cA$ looks as:
		
		
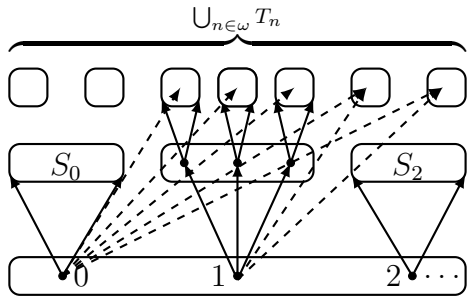
\begin{center}
			\begin{tikzpicture}
				\filldraw[fill=white, draw=black, thick,rounded corners] (0,0) rectangle (6,.5);
				
				\draw[fill=black] (3,.25) circle [radius=0.05];
				\node [left ] at (3,.25) {$1$};
				\filldraw[fill=white, draw=black, thick,rounded corners] (2,2) rectangle (4,1.5);
				
				\draw[fill=black] (3,1.75) circle [radius=0.05];
				\node [left ] at (3,1.75) {};
				\draw[black, thick, -latex] (3,.25) -- (3,1.75);
				
				\draw[black, thick, -latex] (3,1.75) -- (2.8,2.6);
				\draw[black, thick, -latex] (3,1.75) -- (3.2,2.6);

				\draw[fill=black] (2.3,1.75) circle [radius=0.05];
				\draw[black, thick, -latex] (2.3,1.75) -- (2,2.6);
				\draw[black, thick, -latex] (2.3,1.75) -- (2.5,2.6);

				\draw[fill=black] (3.7,1.75) circle [radius=0.05];
				\draw[black, thick, -latex] (3.7,1.75) -- (4,2.6);
				\draw[black, thick, -latex] (3.7,1.75) -- (3.5,2.6);

				\draw[black, thick, -latex] (3,.2) -- (3.7,1.75);			
				\draw[black, thick, -latex] (3,.2) -- (2.3,1.75);


				\draw[fill=black] (5.3,.25) circle [radius=0.05];
				\node [left ] at (5.3,.25) {$2$};			
				
				\node [right ] at (5.3,.25) {$\dots$};			
				\filldraw[fill=white, draw=black, thick,rounded corners] (0,2) rectangle (1.5,1.5);
				\node [left ] at (1.1,1.7) {$S_{0}$};

				\draw[fill=black] (.7,.25) circle [radius=0.05];
				\node [right ] at (.7,.25) {${0}$};
				
				\draw[black, thick, -latex] (.7,.25) -- (1.5,1.6);				
				\draw[black, thick, -latex] (.7,.25) -- (0,1.6);

				\filldraw[fill=white, draw=black, thick,rounded corners] (4.5,2) rectangle (6,1.5);
				\node [left ] at (5.6,1.7) {$S_{2}$};
				
				\draw[black, thick, -latex] (5.3,.25) -- (4.5,1.6);				
				\draw[black, thick, -latex] (5.3,.25) -- (6,1.6);

				\filldraw[fill=white, draw=black, thick,rounded corners] (2,2.5) rectangle (2.5,3);

				\filldraw[fill=white, draw=black, thick,rounded corners] (3.5,2.5) rectangle (4,3);
				
				\filldraw[fill=white, draw=black, thick,rounded corners] (2.75,2.5) rectangle (3.25,3);

				\filldraw[fill=white, draw=black, thick,rounded corners] (0,2.5) rectangle (.5,3);
				\node [] at (0.25,2.7) {};
				
				\filldraw[fill=white, draw=black, thick,rounded corners] (1.5,2.5) rectangle (1,3);
				
				\filldraw[fill=white, draw=black, thick,rounded corners] (2.75,2.5) rectangle (3.25,3);

				\filldraw[fill=white, draw=black, thick,rounded corners] (4.5,2.5) rectangle (5,3);

				\filldraw[fill=white, draw=black, thick,rounded corners] (5.5,2.5) rectangle (6,3);
				
				\node[black, below] at (3,4) {		$ \overbrace{\hspace{6cm}}^{\bigcup_{n\in\omega}T_n} $};

				\draw[black, thick,dashed, -latex] (.7,.25) -- (3,2.75);				
				\draw[black, thick,dashed,-latex] (.7,.25) -- (2.25,2.75);				
				\draw[black, thick,dashed, -latex] (.7,.25) -- (3.75,2.75);				
				
				\draw[black, thick,dashed,-latex] (.7,.25) -- (4.7,2.75);				
				\draw[black, thick,dashed, -latex] (.7,.25) -- (5.7,2.75);

				\draw[black, thick,dashed, -latex] (3,.2) -- (4.7,2.75);		
				\draw[black, thick,dashed, -latex] (3,.2) -- (5.7,2.75);

			\end{tikzpicture}
		\end{center}
		Let $\varphi:= \exists x\exists y \exists z(R(x,y)\wedge R(y,z)\wedge R(x,z))$. Trivially, $\cA\not\models\varphi$, but one can construct  ultrafilters $u,v,z\in \Uf(A)$, for which $\cA^\ue\not\models \varphi$.
	\end{example}
	
	
	
	
	{}

\end{document}